\documentclass[review]{elsarticle}
\usepackage{bm,amssymb,amsmath,mathrsfs}
\usepackage{amsthm}
\usepackage{lineno}
\usepackage[usenames]{color}

\usepackage{dcolumn}% Align table columns on decimal point

\theoremstyle{plain}
\newtheorem{theorem}{Theorem}[section]
\newtheorem{remark}[theorem]{Remark}
\newtheorem{proposition}[theorem]{Proposition}

\begin{document}

\journal{(internal report CC24-8)}

\begin{frontmatter}

\title{Solutions for $k$-generalized Fibonacci numbers using Fuss-Catalan numbers}

\author[cc]{S.~R.~Mane}
\ead{srmane001@gmail.com}
\address[cc]{Convergent Computing Inc., P.~O.~Box 561, Shoreham, NY 11786, USA}

\begin{abstract}
  We present new expressions for the $k$-generalized Fibonacci numbers, say $F_k(n)$.
  They satisfy the recurrence $F_k(n) = F_k(n-1) +\dots+F_k(n-k)$.
  Explicit expressions for the roots of the auxiliary (or characteristic) polynomial are presented, using Fuss-Catalan numbers.
  Properties of the roots are enumerated.
  We quantify the accuracy of asymptotic approximations for $F_k(n)$ for $n\gg1$.
  Our results subsume and extend some results published by previous authors.
  We also present a basis (or `fundamental solutions') to solve the above recurrence for arbitrary initial conditions.  
  We comment on the use of generating functions and multinomial sums for the $k$-generalized Fibonacci numbers and related sequences.
  We note that the resulting multinomial sums are Dickson polynomials of the second kind in several variables.
  We also present what may be a new identity for companion matrices.
\end{abstract}

\begin{keyword}
  Generalized Fibonacci numbers
  \sep Fuss-Catalan numbers
  \sep recurrences
  \sep Vandermonde systems

%\vskip 0.25in
% MSC2020 codes here, in the form: \MSC code \sep code
\MSC[2020]{
%primary 
11B39  % Fibonacci and Lucas numbers
\sep 11B37  % recurrences
\sep 39A06  % linear difference equations
}

\end{keyword}

\end{frontmatter}

%\newpage
\setcounter{equation}{0}
\section{Introduction}\label{sec:intro}
The origin of this paper is somewhat curious.
The author coauthored a paper \cite{DM3} on the success run waiting times of Bernoulli trials.
Recently, the author noticed that, for the special case $p=\frac12$ ($p$ being the success parameter of the Bernoulli trials),
elementary transformations yielded the solution for the so-called `$k$-generalized' Fibonacci numbers, say $F_k(n)$.
These numbers satisfy the recurrence
\begin{equation}
\label{eq:rec_Fnk}
F_k(n) = F_k(n-1) +\dots +F_k(n-k) \,.
\end{equation}
The usual Fibonacci numbers are the special case $k=2$.

Our main contribution is to display
explicit series solutions for all the roots of the characteristic equation of the above recurrence, for arbitrary $k\ge2$
(leveraging formulas from \cite{DM3}).
Solutions in radicals are known for special cases such as the usual Fibonacci, tribonacci and tetranacci numbers.
It has long been known that the characteristic equation has a unique positive real root (the `principal root'),
whose magnitude exceeds that of all the other roots.
Some authors, e.g.~Wolfram \cite{Wolfram}, have derived expressions for the principal root and/or upper and lower bounds for it.
We present expressions for \emph{all} the roots and moreover enumerate several properties of the roots.
For example,
(i) there are no pure imaginary roots,
and (ii) the smaller the real part of a root (less positive or more negative), the smaller its amplitude.
We solve an associated Vandermonde system of equations and derive a Binet-style formula for the solution of the recurrence.
(Other authors have also published the same or similar Binet-style formulas,
e.g.~Dresden and Du \cite{DresdenDu}, but the roots had to be computed numerically, for example via a root-finding algorithm.)

We then apply our formalism to treat related problems published by other authors.
For example, Spickerman and Joyner \cite{SpickermanJoyner} solved the recurrence eq.~\eqref{eq:rec_Fnk},
but with the initial values $1,1,2,4,8,\dots$.
We generalize from powers of $2$ to powers of an arbitrary number $\mu$.
Our solution is presented in Sec.~\ref{sec:SJ}.

Next, in Sec.~\ref{sec:basis_seq}
we present a set of `basis sequences' or `fundamental solutions' of the recurrence eq.~\eqref{eq:rec_Fnk}.
The usual scenario is to solve eq.~\eqref{eq:rec_Fnk} only for the initial conditions $F_k(0)=\dots=F_k(k-2)=0$ and $F_k(k-1)=1$,
i.e.~the initial $k$-tuple $(0,\dots,0,1)$.
We present solutions for the other initial $k$-tuples $(1,0,\dots,0)$ through $(0,\dots,0,1,0)$.
To the author's knowledge, such a set of basis sequences has not been published in the literature.
We leverage work from papers on solutions of linear recurrences, in particular by Wolfram \cite{Wolfram2000}.

Sec.~\ref{sec:compmtx} reviews the use of a companion matrix to solve linear recurrences with constant coefficients,
mainly to make contact with the material in Sec.~\ref{sec:mult_sum}.
In this context, employing a lemma by Wolfram \cite{Wolfram2000},
we present what may be a new identity for companion matrices.
See Prop.~\ref{thm:ident_cmpmtx}.

Finally, in Sec.~\ref{sec:mult_sum} we remark on the use of generating functions and multinomial sums to solve linear recurrences with constant coefficients.
In this case, our presentation is more of an independent confirmation of known results,
to demonstrate the application of our formalism for basis sequences.
We treat the Narayana, Padovan and Perrin sequences.
A referee kindly brought to the author's attention that the multinomial sums in Sec.~\ref{sec:mult_sum}
are Dickson polynomials of the second kind in several variables.

%\newpage
\setcounter{equation}{0}
\section{Roots of the characteristic equation}\label{sec:fusscat}
Eq.~\eqref{eq:rec_Fnk} is a homogeneous linear difference equation and its characteristic equation is
\begin{equation}
\label{eq:aux_eq}
x^k - \sum_{j=0}^{k-1}x^j = 0 \,.
\end{equation}
The corresponding auxiliary (or characteristic) polynomial is
$A(x) = x^k -x^{k-1} -\dots -1$.
Several authors \cite{Wolfram,Miles,Miller} have shown that the roots of $A(x)$ are simple.
Let us denote the roots of $A(x)$ by $\zeta_j$, $j=0,\dots,k-1$.
Then we express $F_k(n)$ as a weighted sum of powers of the roots as follows
(i.e.~a Binet-style formula)
\begin{equation}
\label{eq:Fn_sum_roots}
F_k(n) = \sum_{j=0}^{k-1} c_j\zeta_j^n \,.
\end{equation}
Here $\{c_0,\dots,c_{k-1}\}$ is a set of coefficients which do not depend on $n$,
but they do depend on $k$ and the initial conditions, i.e.~the values of $F_k(0)$ through $F_k(k-1)$.
It is standard to employ the initial conditions $F_k(0)=\dots=F_k(k-2)=0$ and $F_k(k-1)=1$.
We shall begin with this, but later we shall treat alternative initial conditions,
by Spickerman and Joyner \cite{SpickermanJoyner}.
In Sec.~\ref{sec:basis_seq}, we present a complete set of `fundamental solutions' for the recurrence eq.~\eqref{eq:rec_Fnk},
i.e.~a basis which can be used to solve for arbitrary initial conditions.

Our main result is to present explicit expressions for the roots $\zeta_j$,
in terms of so-called `Fuss-Catalan numbers'
(see the text by Graham, Knuth and Patashnik \cite{GrahamKnuthPatashnik}
for relevant definitions, formulas and identities).
The connection is somewhat surprising.
Dilworth and Mane \cite{DM3} (hereafter `DM') published a paper on the success run waiting times of Bernoulli trials.
The recurrence for the probability mass function $f_k(n)$
is (\cite{DM3} eq.~(1)), with $p\in(0,1)$, $q=1-p$ and $k\ge2$,
\begin{equation}
\label{eq:rec_DM3}
f_k(n) = qf_k(n-1) +pqf_k(n-2) +p^2qf_k(n-3) +\cdots +p^{k-1}qf_k(n-k) \,.
\end{equation}
Set $p=q=\frac12$ and $f_k(n) = F_k(n)/2^n$; then elementary manipulations yield eq.~\eqref{eq:rec_Fnk}.
Hence the solution for the $k$-generalized Fibonacci numbers $F_k(n)$
equals $2^n$ times the probability mass function $f_k(n)$ of success run waiting times of Bernoulli trials,
with $p=q=\frac12$.
The auxiliary polynomial associated with eq.~\eqref{eq:rec_DM3} is
$\mathcal{A}_p(x) = x^k -qx^{k-1} -pqx^{k-2} -\cdots -p^{k-1}q$.
Barry and Lo Bello \cite{BarryLoBello} showed that the roots of $\mathcal{A}_p(x)$ are simple,
which goes to show that there are fruitful connections between different fields of research.
Let us denote the roots of $\mathcal{A}_p(x)$ by $\lambda_j$, $j=0,\dots,k-1$.
Then $\zeta_j=2\lambda_j$, with $p=q=\frac12$.
The following properties of the roots were listed by DM (\cite{DM3} Prop.~13).
\begin{enumerate}
\item
It is known that $\mathcal{A}_p(x)$ has a unique positive root, which has a larger magnitude than all the other roots \cite{Feller}.
We denote it by $\lambda_0$ below.
\item
If $k$ is odd, there are no other real roots. If $k$ is even, there is exactly one real negative root.
\item
For all $k$, the remaining roots come in complex conjugate pairs.
\item
There are no pure imaginary roots.
\end{enumerate}
Wolfram \cite{Wolfram} proved the first three results above for the case of the $k$-generalized Fibonacci numbers.
The last item may not be so well-known.
DM employed the terms `pricipal root' for $\lambda_0$
and `secondary roots' for the other $\lambda_j$, $j=1,\dots,k-1$.

DM obtained expressions for the roots $\lambda_j$ as power series,
whose coefficients are Fuss-Catalan numbers (\cite{DM3} Theorem 1).
They identified three cases, viz.~(i) $p<k/(k+1)$, (ii) $p=k/(k+1)$ and (iii) $p>k/(k+1)$.
We require only the first case, because we fix $p=\frac12$ and $k/(k+1)>\frac12$ for all $k>1$.
First define, for $n\ge1$,
\begin{equation}
\label{eq:def_bn}
b_n = \frac{\Gamma(n-1+n/k)}{n!\Gamma(n/k)} \,.
\end{equation}
The Fuss-Catalan numbers are defined as follows (\cite{DM3} eq.~(9))
\begin{equation}
\label{eq:FussCat_def}
A_m(\nu,r) = \frac{r}{\Gamma(m+1)}\frac{\Gamma(m\nu+r)}{\Gamma(m(\nu-1)+r+1)}\,.
\end{equation}
Then $b_n$ is a Fuss-Catalan number (\cite{DM3} eq.~(10))
\begin{equation}
b_n = -A_n(1+1/k, -1)\,.
\end{equation}
The principal root $\zeta_0$ is given by (\cite{DM3} eq.~(8))
\begin{equation}
\label{eq:DM_prin}
  \zeta_0 = 2 - 2k\sum_{m=1}^\infty b_{mk} \, (p^kq)^m
  = 2 - 2k\sum_{m=1}^\infty \frac{b_{mk}}{2^{m(k+1)}} \,.
\end{equation}
The secondary roots $\zeta_j$, $j=1,\dots,k-1$ are given by (\cite{DM3} eq.~(7))
\begin{equation}
\label{eq:secroots}
\zeta_j = 2\sum_{n=1}^\infty b_n \, (e^{2\pi ij/k}pq^{1/k})^n 
= 2\sum_{n=1}^\infty b_n \, \biggl(\frac{e^{2\pi ij/k}}{2^{(k+1)/k}}\biggr)^n \,. 
\end{equation}
\begin{proposition}  
Wolfram (\cite{Wolfram} Theorem 3.9) derived the following sum for the principal root.
The root was expressed as $\zeta_0 = 2(1-\varepsilon_k)$ and
\begin{equation}
\label{eq:W_vareps_k}
\varepsilon_k = \sum_{i\ge1} \binom{(k+1)i-2}{i-1}\frac{1}{i2^{(k+1)i}} \,.
\end{equation}
Wolfram's solution is equivalent to our expression in eq.~\eqref{eq:DM_prin}.
\end{proposition}
\begin{proof}
From eq.~\eqref{eq:def_bn}, note that
\begin{equation}
b_{mk} = \frac{\Gamma(mk-1+m)}{(mk)!\Gamma(m)} = \frac{(mk+m-2)!}{(mk!)(m-1)!} = \frac{1}{mk}\binom{(k+1)m-2}{m-1} \,.
\end{equation}
Next using eq.~\eqref{eq:DM_prin}, we deduce
\begin{equation}
\begin{split}
\varepsilon_k &= k\sum_{m=1}^\infty \frac{b_{mk}}{2^{(k+1)m}} 
\\
&= k\sum_{m=1}^\infty \frac{1}{mk}\binom{(k+1)m-2}{m-1}\frac{1}{2^{(k+1)m}} 
\\
&= \sum_{m=1}^\infty \binom{(k+1)m-2}{m-1}\frac{1}{m2^{(k+1)m}} \,.
\end{split}
\end{equation}
This equals eq.~\eqref{eq:W_vareps_k}, substitute $i$ for $m$.
\end{proof}
However, Wolfram did not publish expressions for the secondary roots.
Spickerman and Joyner (\cite{SpickermanJoyner} Sec.~3) computed the roots numerically 
(to $4$ decimal places) for $2\le k\le 10$.
Using eqs.~\eqref{eq:DM_prin} and \eqref{eq:secroots}, we confirm their numbers are correct.

It remains now to derive expressions for $F_k(n)$ for specific initial conditions,
and also to derive asymptotic expressions for $n\gg1$
(in which case the solution is dominated by the contribution from the principal root).
To do so, it will be helpful below to work with the polynomial $B(x) = (x-1)A(x)$ for the following reason.
Note that
\begin{equation}
\label{eq:A_frac}
\begin{split}
A(x) &= x^k -x^{k-1} -\dots -1
\\
&= x^k - \frac{x^k-1}{x-1}
\\
&= \frac{x^{k+1}-2x^k+1}{x-1} \,.
\end{split}
\end{equation}
Hence $B(x) = x^{k+1}-2x^k+1$ is a trinomial, which is easier to manipulate. However, it contains an extraneous root $x=1$.
Note that $A(1)=1-k\ne0$, i.e.~$x=1$ is not a root of $A(x)$, hence $B(x)$ has no repeated roots.

%\newpage
\setcounter{equation}{0}
\section{Properties of roots}\label{sec:proproots}
We may read off several properties of the roots from DM \cite{DM3}.
First, the principal root lies in the interval (\cite{DM3} Props.~8 and 11)   
\begin{equation}
\label{eq:DM_bounds_prin}
2 - \frac{2}{k+1} < \zeta_0 < 2 - \frac{1}{2^k} \,.
\end{equation}
Wolfram (\cite{Wolfram} Lemma 3.6) derived the bounds 
\begin{equation}
\label{eq:W_bounds_prin}
2(1-2^{-k}) < \zeta_0 < 2 \,.
\end{equation}
Eq.~\eqref{eq:DM_bounds_prin} has a tighter upper bound and
eq.~\eqref{eq:W_bounds_prin} has a tighter lower bound.
Combining them, we obtain the improved bounds
\begin{equation}
\label{eq:better_bounds_prin}
2 - \frac{1}{2^{k-1}} < \zeta_0 < 2 - \frac{1}{2^k} \,.
\end{equation}
Next, DM (\cite{DM3} Prop.~14) showed that the secondary roots satisfy $|\lambda_j|<p$.
Setting $p=\frac12$, this implies $|\zeta_j| < 1$ for $j=1,\dots,k-1$.
Wolfram (\cite{Wolfram} Lemma 3.6) obtained a tighter lower bound
\begin{equation}
\label{eq:W_bounds_sec}
  3^{-1/k} < |\zeta_j| < 1 \,.
\end{equation}
This corrects a misprint in \cite{Wolfram}, where the lower bound is given as $3^{-k}$
(but the value $3^{-1/k}$ is stated correctly in the \emph{proof} of Lemma 3.6 in \cite{Wolfram}).
Next, the relative magnitudes of \emph{distinct} roots are as follows (\cite{DM3} Prop.~16)
\begin{enumerate}
\item
  $|\zeta_{j_1}| = |\zeta_{j_2}|$ if and only if $\zeta_{j_1} = \bar{\zeta}_{j_2}$.
\item
  $\Re(\zeta_{j_1}) < \Re(\zeta_{j_2})$ if and only if $|\zeta_{j_1}| < |\zeta_{j_2}|$.  
\item
  For even $k$, the negative real root has the most negative
  real part, hence the smallest amplitude of all the roots.
\end{enumerate}
Spickerman and Joyner (\cite{SpickermanJoyner} Sec.~3) computed the roots numerically 
for $2\le k\le 10$ and their numbers satisfy the above relations.

\emph{Distribution of the roots in the complex plane.}
Observe the following limits as $k\to\infty$:
(i) for the principal root, $\limsup_{k\to\infty}\zeta_0=2$ (from eq.~\eqref{eq:better_bounds_prin}) and
(ii) for all the secondary roots, $\limsup_{k\to\infty}|\zeta_j|=1$ (from eq.~\eqref{eq:W_bounds_sec}).
Hence all the secondary roots approach the unit circle (from the inside),
but the principal root approaches $2$ (from below).
A graph of the roots in the complex plane is plotted in Fig.~\ref{fig:roots_k40}, for $k=40$.
For $k\gg1$, the secondary roots are also distributed approximately uniformly around the unit circle.
A graph of the arguments of the secondary roots, i.e.~$\arg(\zeta_j)/(2\pi)$,
is plotted in Fig.~\ref{fig:arg_roots_k40}, for $k=40$.
The points lie close to a straight line.

%\newpage
\setcounter{equation}{0}
\section{Vandermonde}\label{sec:vdm}
First recall eq.~\eqref{eq:Fn_sum_roots}.
We set up a Vandermonde system of equations to solve for the coefficients $c_j$ as follows
\begin{equation}
\label{eq:vdm_eqns}
\begin{split}
\begin{pmatrix} 
1 & 1 & \dots & 1 \\
\zeta_0 & \zeta_1 & \dots & \zeta_{k-1} \\
\zeta_0^2 & \zeta_1^2 & \dots & \zeta_{k-1}^2 \\
\vdots &&& \vdots \\
\zeta_0^{k-1} & \zeta_1^{k-1} & \dots & \zeta_{k-1}^{k-1} 
\end{pmatrix}
\begin{pmatrix} 
c_0 \\
c_1 \\
c_2 \\
\vdots \\
c_{k-1} 
\end{pmatrix}
=
\begin{pmatrix} 
F_k(0) \\
F_k(1) \\
F_k(2) \\
\vdots \\
F_k(k-1)
\end{pmatrix} \,.
\end{split}
\end{equation}
We begin with the initial conditions $F_k(0)=\dots=F_k(k-2)=0$ and $F_k(k-1)=1$.
(We treat the initial condition by Spickerman and Joyner \cite{SpickermanJoyner} in Sec.~\ref{sec:SJ}.)
The solution is (\cite{DM3}, eq.~(60))
\begin{equation}
\label{eq:sol_cell}
c_j = \frac{1}{\prod_{s\ne j}(\zeta_j-\zeta_s)} \,.
\end{equation}
Note that
\begin{equation}
B(x) = (x-1)\prod_{s=0}^{k-1}(x-\zeta_s) \,.
\end{equation}
Hence
\begin{equation}  
B^\prime(x) = \frac{B(x)}{x-1} +\sum_{s=0}^{k-1} \frac{B(x)}{x-\zeta_s} \,.
\end{equation}
Then, noting that $B(\zeta_j)=0$,
\begin{equation}  
B^\prime(\zeta_j) = (\zeta_j-1)\prod_{s\ne j}(\zeta_j-\zeta_s) \,.
\end{equation}
Also $B^\prime(x) = (k+1)x^k-2kx^{k-1}$, thus $B^\prime(\zeta_j) = ((k+1)\zeta_j-2k)\zeta_j^{k-1}$.
Hence
\begin{equation}
\begin{split}
  c_j &= \frac{1}{\prod_{s\ne j}(\zeta_j-\zeta_s)}
  = \frac{\zeta_j-1}{B^\prime(\zeta_j)}
  = \frac{\zeta_j-1}{((k+1)\zeta_j-2k)\zeta_j^{k-1}} \,.
\end{split}
\end{equation}
Hence the solution for $F_k(n)$ is
\begin{equation}
\label{eq:sol_Fnk}
F_k(n) = \sum_{j=0}^{k-1} \frac{\zeta_j-1}{(k+1)\zeta_j-2k} \zeta_j^{n-k+1} \,.
\end{equation}
This is our expression for the $k$-generalized Fibonacci numbers, as a sum over powers of the roots of the auxiliary polynomial.
The expressions for special cases such as the tribonacci numbers ($k=3$) and tetranacci numbers ($k=4$) are known in radicals.
However eq.~\eqref{eq:sol_Fnk} presents a unified formula for all $k\ge2$.
The use of Fuss-Catalan numbers furnishes a systematic formula to compute the roots to arbitrary accuracy.
\begin{remark}
  The Binet-style sum in eq.~\eqref{eq:sol_Fnk} also works for $n<0$.
  The recurrence can be run backwards to negative values of $n$.
\end{remark}
\begin{remark}
An expression equivalent to eq.~\eqref{eq:sol_Fnk} was published by Dresden and Du (\cite{DresdenDu} Theorem 1),
where the roots were denoted by $\alpha_j$, $j=1,\dots,k$.
However, they did not exhibit expressions for the roots.
They denoted the $k$-generalized Fibonacci numbers by $F_n^{(k)}$, with the recurrence
$F_n^{(k)} = F_{n-1}^{(k)} +\dots +F_{n-k}^{(k)}$, but with the initial conditions $F_n^{(k)}=0$ for $n<1$ and $F_n^{(k)}=1$ for $n=1$.
The relation to our notation is $F_k(n) = F_{n-k+2}^{(k)}$.
\end{remark}

%\newpage
\setcounter{equation}{0}
\section{Asymptotic solution}\label{sec:asymp}
For fixed $k$ and $n\gg k$, the asymptotic solution is given by retaining only
the principal root $\zeta_0$ in eq.~\eqref{eq:sol_Fnk}.
The details of the analysis are given in DM (\cite{DM3} Corollary 4).
Use $B(\zeta_0)=0$, i.e.~$\zeta_0^k(\zeta_0-2)+1=0$ to deduce $\zeta_0^{-k} = 2-\zeta_0$.
Then (c.f.~\cite{DM3}, eq.~(16), case $p\ne k/(k+1)$)
\begin{equation}
\begin{split}
F_k(n) &\asymp \frac{\zeta_0-1}{(k+1)\zeta_0-2k} \zeta_0^{n-k+1}
\\
&= \frac{(2-\zeta_0)(\zeta_0-1)}{(k+1)\zeta_0-2k} \zeta_0^{n+1} \,.
\end{split}
\end{equation}
Following DM, we demand that for fixed $\varepsilon>0$,
the magnitude of the contribution to $F_k(n)$ in eq.~\eqref{eq:sol_Fnk}
from all the secondary roots combined
is less than $\varepsilon$ times the contribution from the principal root.
This is achieved if $n\ge N(p,k)$, where (\cite{DM3}, eq.~(17))
\begin{equation}
N = k +\frac{\ln(\varepsilon/(k-1))}{\ln\kappa} \,.
\end{equation}
In our case $\kappa$ is given by (\cite{DM3}, eq.~(18), case $1/(k+1) < p < k/(k+1)$)
\begin{equation}
\kappa = \min\biggl\{\frac{k+1}{2k},\, 1-\delta\biggr\} \,.
\end{equation}
Here (\cite{DM3}, eq.~(19))
\begin{equation}
\delta = \frac{1}{(k+1)^{1/k}}\biggl[ 1 + \frac{1}{2^{(k+1)/k}k} -\sqrt{1+\frac{1}{2^{2(k+1)/k}k^2} +\frac{\cos(2\pi/k)}{2^{1/k}k} }\biggr] \,.
\end{equation}
Then $\delta\to0$ as $k\to\infty$, hence the expression for $\kappa$ simplifies to $\kappa=(k+1)/(2k)$ 
and we obtain the bound
\begin{equation}
N = k +\frac{\ln(\varepsilon/(k-1))}{\ln((k+1)/(2k))} \,.
\end{equation}
\begin{remark}
The above analysis was leveraged from \cite{DM3}.
For the specific application to $k$-generalized Fibonacci numbers, Dresden and Du \cite{DresdenDu} derived a simpler expression.
They showed (\cite{DresdenDu} Theorem 2) that the contribution of all the secondary roots is less than $\frac12$ in magnitude,
and $F_k(n)$ can be obtained from the principal root $\zeta_0$ alone as follows.
We employ our notation in eq.~\eqref{eq:sol_Fnk}, whence
\begin{equation}
\label{eq:Fnk_prin_rnd}
F_k(n) = \left\lfloor \biggl(\frac{\zeta_0-1}{(k+1)\zeta_0-2k} \zeta_0^{n-k+1} +\frac12\biggr) \right\rfloor \,.
\end{equation}
\end{remark}

%\newpage
\setcounter{equation}{0}
\section{Alternative initial conditions: Spickerman and Joyner}\label{sec:SJ}
Spickerman and Joyner \cite{SpickermanJoyner} solved the recurrence eq.~\eqref{eq:rec_Fnk} as follows.
First to avoid confusions of notation, we follow them and write $u_n$ in place of $F_k(n)$.
They employed the initial conditions
$(u_0,\dots,u_{k-1})=(1,1,2,4,\dots,2^{k-2})$, i.e.~$u_0 = 1$ and $u_n = 2^{n-1}$ for $1 \le n \le k-1$.
Hence the next term is $u_k = 1+1+2+4+\dots+2^{k-2} = 2^{k-1}$.
Let us change this slightly and shift the indices down by one (and discard $u_0$) and say instead
$(v_0,\dots,v_{k-1})=(1,2,4,\dots,2^{k-1})$, i.e.~$v_n = 2^n$ for $0\le n \le k-1$.
Let us generalize this even more and replace $2$ by $\mu$,
and write $w_n = \mu^n$ for $0\le n \le k-1$.
Then the solution for $c_j$ in the Vandermonde matrix equations in Sec.~\ref{sec:vdm} is easy;
it is the Lagrange polynomial $c_j = L_j(\mu)$.
The Lagrange polynomial $L_j(x)$ is
\begin{equation}
\label{eq:Lpoly}
L_j(x) = \frac{(x-\zeta_0)\dots(x-\zeta_{j-1})(x-\zeta_{j+1})\dots(x-\zeta_{k-1})}{(\zeta_j-\zeta_0)\dots
  (\zeta_j-\zeta_{j-1})(\zeta_j-\zeta_{j+1})\dots(\zeta_j-\zeta_{k-1})} \,.
\end{equation}
This can be obtained from the auxiliary polynomial, written in the form
(see eq.~\eqref{eq:A_frac})
$A(x) = (x^{k+1}-2x^k+1)/(x-1)$.
The denominator of $L_j(x)$ is $D_j = A^\prime(\zeta_j)$ and does not depend on $x$.
The numerator of $L_j(\mu)$ is $N_j(\mu) = A(\mu)/(\mu-\zeta_j)$, provided $\mu$ is not a root of $A(x)$.
For now, we assume $\mu$ is not a root of $A(x)$.
Then the numerator is
\begin{equation}
N_j(\mu) = \frac{\mu^{k+1}-2\mu^k+1}{(\mu-1)(\mu-\zeta_j)} \,.
\end{equation}
The derivative of $A(x)$ is
\begin{equation}
A^\prime(x) = \frac{(k+1)x^k-2kx^{k-1}}{x-1} -\frac{x^{k+1}-2x^k+1}{(x-1)^2} \,.
\end{equation}
The denominator is given by (note that the second term in $A^\prime(x)$ vanishes for $x=\zeta_j$)
\begin{equation}
\label{eq:denom_Lpoly}
D_j = \frac{(k+1)\zeta_j^k-2k\zeta_j^{k-1}}{\zeta_j-1} = \frac{(k+1)\zeta_j-2k}{\zeta_j-1}\zeta_j^{k-1} \,.
\end{equation}
Hence the coefficient $c_j$ is 
\begin{equation}
\label{eq:cell_SJ}  
c_j = \frac{N_j(\mu)}{D_j} = \frac{\mu^{k+1}-2\mu^k+1}{(\mu-1)(\mu-\zeta_j)} \frac{\zeta_j-1}{(k+1)\zeta_j-2k}\frac{1}{\zeta_j^{k-1}} \,.
\end{equation}
For $\mu=1$, set $\mu=1+\varepsilon$ and take the limit $\varepsilon\to0$ to obtain 
\begin{equation}
\lim_{\mu\to1} c_j = \frac{1-k}{(1-\zeta_j)} \frac{\zeta_j-1}{(k+1)\zeta_j-2k}\frac{1}{\zeta_j^{k-1}}
= \frac{k-1}{(k+1)\zeta_j-2k}\frac{1}{\zeta_j^{k-1}} \,.
\end{equation}
\emph{What if $\mu$ is a root of $A(x)$?}
Actually this is easy.
If $\mu=\zeta_i$ then by construction of the Lagrange polynomials,
$L_i(\zeta_i)=1$ and $L_j(\zeta_i)=0$ for $i\ne j$.
Hence $c_i=1$ and $c_j=0$ for $i\ne j$, i.e.~$c_j = \delta_{ij}$.
\begin{proposition}  
Our expression in eq.~\eqref{eq:cell_SJ}, evaluated for $\mu=2$, is equivalent to Spickerman and Joyner's solution \cite{SpickermanJoyner}.
\end{proposition}  
\begin{proof}  
Spickerman and Joyner treated only the case $\mu=2$.
We know that $A(2)=1$, hence $2$ is not a root of $A(x)$, whence
\begin{equation}
\label{eq:SJ_cj}
\begin{split}
  c_j &= \frac{2^{k+1}-2^{k+1}+1}{(2-1)(2-\zeta_j)} \frac{\zeta_j-1}{(k+1)\zeta_j-2k}\frac{1}{\zeta_j^{k-1}}
  \\
  &= \frac{1}{1/\zeta_j^k} \frac{\zeta_j-1}{(k+1)\zeta_j-2k}\frac{1}{\zeta_j^{k-1}}
  \\
  &= \zeta_j \frac{\zeta_j-1}{(k+1)\zeta_j-2k} \,.
\end{split}
\end{equation}
The expression for $u_n$ by Spickerman and Joyner is (unnumbered, second last line on p.~328 of \cite{SpickermanJoyner})
\begin{equation}
\label{eq:SJ_un}
u_n = \sum_{j=0}^{k-1} \frac{(\alpha_j^{k+1}-\alpha_j^k)\alpha_j^n}{2\alpha_j^k - (k+1)} \,.
\end{equation}
They denoted the roots by $\alpha_j$, whereas we say $\zeta_j$.
We thus express eq.~\eqref{eq:SJ_un} in the form $u_n = \sum_{j=0}^{k-1}d_j\zeta_j^n$, where
\begin{equation}
\label{eq:SJ_dj}
\begin{split}
d_j &= \frac{(\zeta_j^{k+1}-\zeta_j^k)}{2\zeta_j^k - (k+1)} 
\\
&= \zeta_j^k \frac{\zeta_j-1}{2\zeta_j^k - (k+1)} \,. 
\end{split}
\end{equation}
Because of the index shift of $n$ by $1$, the relation of $w$ to $u$ is $w_n(\mu=2) = u_{n+1}$.
Our claim is therefore $c_j\zeta_j^n = d_j\zeta_j^{n+1}$, or $c_j = d_j\zeta_j$.
Employing eqs.~\eqref{eq:SJ_cj} and \eqref{eq:SJ_dj} yields
\begin{equation}
\begin{split}
  \frac{c_j-d_j\zeta_j}{\zeta_j(\zeta_j-1)} &= \frac{1}{(k+1)\zeta_j-2k} - \frac{\zeta_j^k}{2\zeta_j^k-(k+1)}
  \\
  &\propto (2\zeta_j^k-(k+1)) - \zeta_j^k((k+1)\zeta_j-2k)
  \\
  &= 2\zeta_j^k-(k+1) - (k+1)\zeta_j^{k+1}+2k\zeta_j^k
  \\
  &= 2\zeta_j^k-2(k+1)\zeta_j^k+2k\zeta_j^k
  \\
  &= 0 \,.
\end{split}
\end{equation}
Our claim is validated.
\end{proof}  
Note that although Spickerman and Joyner \cite{SpickermanJoyner} gave a correct formal expression for the coefficient $d_j$,
they did \emph{not} offer an explicit expression for the roots of the auxiliary polynomial.
Instead they solved the auxiliary equation numerically for $k=2,\dots,10$ and tabulated the (approximate) numerical values of the roots (\cite{SpickermanJoyner} Sec.~3).
As stated above, our expressions using Fuss-Catalan numbers match their numbers.

%\newpage
\setcounter{equation}{0}
\section{Basis sequences}\label{sec:basis_seq}
\subsection{$k$-generalized Fibonacci}
The $k$-generalized Fibonacci numbers $F_k(n)$ are the solution
of the recurrence eq.~\eqref{eq:rec_Fnk} for only one set of initial conditions,
viz.~$F_k(0)=\dots=F_k(k-2)=0$ and $F_k(k-1)=1$.
We would like a set of $k$ `basis sequences' or `fundamental solutions'
$\{B_{k,i}(n),\,i=0,\dots,k-1\}$
of the recurrence eq.~\eqref{eq:rec_Fnk},
where the sequence $B_{k,i}(n)$ satisfies the initial conditions
$B_{k,i}(n) = \delta_{in}$ for $n=0,\dots,k-1$.
The $k$-generalized Fibonacci numbers $F_k(n)$
are the last member of the basis, viz.~$F_k(n) = B_{k,k-1}(n)$.

Wolfram \cite{Wolfram2000} solved the problem in a more general setting.
Wolfram treated the homogeneous linear recurrence (\cite{Wolfram2000}, eq.~(2.1))
\begin{equation}
\label{eq:W_rec_orig}
f(n) = \sum_{j=1}^k a_{k-j} f(n-j) \,.
\end{equation}
Here $\{a_0,\dots,a_{k-1}\}$ are a set of coefficients which do not depend on $f$ but \emph{can} depend on $n$.
Let us define $\mathscr{W}_i(n)$ to be the solution of
eq.~\eqref{eq:W_rec_orig} with the initial values
$\mathscr{W}_i(n) = \delta_{in}$ for $n=0,\dots,k-1$.
(For brevity, we omit explicit mention of $k$ and $a_0,\dots,a_{k-1}$.)
The set of sequences $\{\mathscr{W}_i(n),\,i=0,\dots,k-1\}$ are a set of
fundamental solutions of the recurrence eq.~\eqref{eq:W_rec_orig}.
Wolfram proved that (\cite{Wolfram2000} Lemma 2.2) 
\begin{equation}
\label{eq:W_basis_orig}
\mathscr{W}_i(n) = \sum_{j=0}^i a_{i-j} \mathscr{W}_{k-1}(n-j-1) \qquad\qquad (i=0,\dots,k-2) \,.
\end{equation}
\emph{Hence it is only necessary to solve the recurrence eq.~\eqref{eq:W_rec} for the last basis function $\mathscr{W}_{k-1}(n)$.}
All the other basis functions can be obtained from it. This is a key insight.

We treat only constant coefficients in this paper.
Let $\{\beta_1,\dots,\beta_k\}$ be a set of constant coefficients and reexpress eq.~\eqref{eq:W_rec_orig} as follows.
(Note the indexing of the $\beta_j$, to maintain consistency with other usage of indexing in this paper.)
\begin{equation}
\label{eq:W_rec}
f(n) = \sum_{j=1}^k \beta_j f(n-j) \,.
\end{equation}
We define $W_i(n)$ to be the solution of
eq.~\eqref{eq:W_rec} with the initial values
$W_i(n) = \delta_{in}$ for $n=0,\dots,k-1$.
Then, from eq.~\eqref{eq:W_basis_orig},
\begin{equation}
\label{eq:W_basis}
W_i(n) = \sum_{s=0}^i \beta_{k+s-i} W_{k-1}(n-s-1) \qquad\qquad (i=0,\dots,k-2) \,.
\end{equation}
We shall employ $W_i(n)$ and eq.~\eqref{eq:W_basis} in the rest of this paper.
\begin{proposition}
The basis sequences for the $k$-generalized Fibonacci numbers are given by
\begin{equation}
\label{eq:Fnk_basis_sum_roots}
B_{k,k-m}(n)
= \sum_{j=0}^{k-1} \frac{\zeta_j^m-2\zeta_j^{m-1}+1}{(k+1)\zeta_j-2k} \zeta_j^{n-k+1} \qquad (m=1,\dots,k) \,.
\end{equation}
\end{proposition}
\begin{proof}
The $k$-generalized Fibonacci recurrence eq.~\eqref{eq:rec_Fnk} is the special case $\beta_1=\dots=\beta_k=1$ of eq.~\eqref{eq:W_rec}.
It follows immediately from eq.~\eqref{eq:W_basis}, where our last basis function is $F_k(n)$, that
\begin{equation}
\label{eq:Fnk_basis}
B_{k,i}(n) = \sum_{s=0}^i F_k(n-s-1) \,.
\end{equation}
We employ the expression in eq.~\eqref{eq:sol_Fnk} and set $i=k-m$.
In the line tagged by $(*)$ below,
we employ the auxiliary equation to deduce $\zeta_j^{-k-1} = 2\zeta_j^{-1}-1$, because $\zeta_j^{k+1}-2\zeta_j^k+1=0$ (see eq.~\eqref{eq:A_frac}).
\begin{equation}
\begin{split}
  B_{k,k-m}(n) &= \sum_{s=0}^{k-m} F_k(n-s-1)
  \\
  &= \sum_{s=0}^{k-m} \sum_{j=0}^{k-1} \frac{\zeta_j-1}{(k+1)\zeta_j-2k} \frac{\zeta_j^{n-k+1}}{\zeta_j^{s+1}}
  \\
  &= \sum_{j=0}^{k-1} \frac{\zeta_j-1}{(k+1)\zeta_j-2k} \zeta_j^{n-k+1} \biggl(\frac{1}{\zeta_j}\sum_{s=0}^{k-m} \frac{1}{\zeta_j^s}\biggr)
  \\
  &= \sum_{j=0}^{k-1} \frac{\zeta_j-1}{(k+1)\zeta_j-2k} \zeta_j^{n-k+1} \frac{1-\zeta_j^{m-k-1}}{\zeta_j-1}
  \\
  &= \sum_{j=0}^{k-1} \frac{1}{(k+1)\zeta_j-2k} \zeta_j^{n-k+1} (1 -\zeta_j^m(2\zeta_j^{-1}-1)) \qquad\qquad (*)
  \\
  &= \sum_{j=0}^{k-1} \frac{1}{(k+1)\zeta_j-2k} \zeta_j^{n-k+1} (\zeta_j^m -2\zeta_j^{m-1} +1) \,.
\end{split}
\end{equation}
Rearranging terms in the numerator yields eq.~\eqref{eq:Fnk_basis_sum_roots}.
\end{proof}
\begin{remark}
Observe that eq.~\eqref{eq:Fnk_basis_sum_roots}
is a Binet-style sum for each value of $m$, and is valid for negative $n$.
The dependence on $m$ has a simple pattern.
The case $m=1$ equals eq.~\eqref{eq:sol_Fnk}.
To the author's knowledge, eq.~\eqref{eq:Fnk_basis_sum_roots} has not been published in the literature.
\end{remark}

Next, given a tuple of initial values $\vec{\gamma}=(\gamma_0,\dots,\gamma_{k-1})$ for $n=0,\dots,k-1$,
the solution $B(n,k,\vec{\gamma})$ of the recurrence eq.~\eqref{eq:rec_Fnk} is given by the linear combination
\begin{equation}
B(n,k,\vec{\gamma}) = \sum_{i=0}^{k-1}\gamma_i B_{k,i}(n) \,.
\end{equation}
We can leave it at this, because eq.~\eqref{eq:Fnk_basis_sum_roots} furnishes expressions for all the basis functions $B_{k,i}(n)$.

However, Wolfram's formalism \cite{Wolfram2000} and eq.~\eqref{eq:W_basis}
furnishes tools to offer a solution for the more general recurrence in eq.~\eqref{eq:W_rec}.
Using eq.~\eqref{eq:W_basis}, we can also express everything in terms of only the last basis function, which is $W_{k-1}(n)$.
Given a tuple of initial values $\vec{\gamma}=(\gamma_0,\dots,\gamma_{k-1})$ for $n=0,\dots,k-1$,
denote the solution of the recurrence eq.~\eqref{eq:W_rec} by $W(n,k,\vec{\gamma})$.
Then
\begin{equation}
\begin{split}
W(n,k,\vec{\gamma}) &= \sum_{i=0}^{k-1}\gamma_i W_i(n)
\\
&= \gamma_0 \beta_k W_{k-1}(n-1)
\\
&\quad
+\gamma_1 \bigl(\beta_{k-1}W_{k-1}(n-1) +\beta_kW_{k-1}(n-2)\bigr)
%\\
%&\quad
% +\gamma_2 \bigl(\beta_{k-2}W_{k-1}(n-1) +\beta_{k-1}W_{k-1}(n-2) +\beta_kW_{k-1}(n-3)\bigr)
\\
&\quad
 +\dots
\\
&\quad
 +\gamma_{k-2} \bigl(\beta_2W_{k-1}(n-1) +\beta_3W_{k-1}(n-2) +\dots +\beta_kW_{k-1}(n-k+1)\bigr)
\\
&\quad
 +\gamma_{k-1} W_{k-1}(n)
\\
&= \bigl(\beta_k\gamma_0 +\beta_{k-1}\gamma_1 +\dots +\beta_2\gamma_{k-2}\bigr) W_{k-1}(n-1) 
\\
&\quad +\bigl(\beta_k\gamma_1 +\beta_{k-1}\gamma_2 +\dots +\beta_3\gamma_{k-2}\bigr) W_{k-1}(n-2)  
\\
&\quad +\dots
\\
&\quad +\beta_k\gamma_{k-2} W_{k-1}(n-k+1) 
\\
&\quad +\gamma_{k-1} W_{k-1}(n)
\\
&= \gamma_{k-1} W_{k-1}(n) +\sum_{i=1}^{k-1} W_{k-1}(n-i) \biggl(\sum_{j=2}^{k-i+1} \beta_{j+i-1}\gamma_{k-j}\biggr) \,.
\end{split}
\end{equation}
The nested sums over $j$ do not depend on $n$, hence they can be precomputed,
say $\delta_i \equiv \sum_{j=2}^{k-i+1} \beta_{j+i-1}\gamma_{k-j}$ for $i=1,\dots,k-1$
and formally define $\delta_0 = \gamma_{k-1}$. Then
\begin{equation}
W(n,k,\vec{\gamma}) = \sum_{i=0}^{k-1} \delta_i W_{k-1}(n-i) \,.
\end{equation}

%\newpage
\setcounter{equation}{0}
\section{Companion matrix}\label{sec:compmtx}
This section was prompted by an insightful comment by a referee.
It contains alternative formalism for the solution of linear recurrences with constant coefficients.
Our exposition follows Chen and Louck \cite{ChenLouck}.
Recall the recurrence in eq.~\eqref{eq:W_rec}, reproduced here for ease of reference.
\begin{equation}
\label{eq:Fnbeta}
f(n) = \beta_1f(n-1) +\cdots +\beta_kf(n-k) \,.
\end{equation}
The characteristic polynomial of the recurrence in eq.~\eqref{eq:Fnbeta} is
$p_k(x) = x^k -\beta_1x^{k-1} -\beta_2x^{k-2} -\dots -\beta_k$.
The companion matrix $C_k$ of $p_k(x)$ is defined as follows (\cite{ChenLouck}, eq.~(2.1))
\begin{equation}
\label{eq:ck}
  C_k = \begin{pmatrix} \beta_1 & \beta_2 & \beta_3 & \dots & \beta_{k-1} & \beta_k \\
    1 & 0 & 0 & \dots & 0 & 0 \\
    0 & 1 & 0 & \dots & 0 & 0 \\
    \vdots & \vdots &\vdots & \dots & \vdots & \vdots \\
    0 & 0 & 0 & \dots & 1 & 0
    \end{pmatrix} \,.
\end{equation}
The characteristic equation of the matrix $C_k$ is $p_k(\lambda)=0$, which is also
the characteristic equation of the recurrence in eq.~\eqref{eq:Fnbeta}.
(For the special case $\beta_1=\dots=\beta_k=1$,
the Fuss-Catalan series for the roots $\zeta_j$, $j=0,\dots,k-1$ in
eqs.~\eqref{eq:DM_prin} and \eqref{eq:secroots}
are the eigenvalues of the companion matrix $C_k$.)
Then (\cite{ChenLouck} Theorem 3.1) the $(i,j)$ entry $c_{i,j}^{(n)}$ in the matrix $C_k^n$ is given by  
(\cite{ChenLouck} eq.~(3.1), note that the indexing is $1 \le i,j \le k$)
\begin{equation}
\label{eq:cijn}
  c_{i,j}^{(n)} = \sum_{i_1+2i_2+\dots+ki_k = n-i+j} \frac{i_j+\dots+i_k}{i_1+\dots+i_k}
  \binom{i_1+i_2+\dots+i_k}{i_1,i_2,\dots, i_k}\,\beta_1^{i_1}\dots\beta_k^{i_k} \,.
\end{equation}
Also, $c_{i,j}^{(n)}$ is defined to equal unity if $n=i-j$.
The companion matrix can be employed to iterate the solution of a
linear recurrence with constant coefficients as follows (\cite{ChenLouck}, Sec.~7).
Iterating $n\ge1$ times yields (\cite{ChenLouck} eq.~(7.3))
\begin{equation}
\begin{pmatrix} f(n+k-1) \\ f(n+k-2) \\ \vdots \\ f(n) \end{pmatrix} = C_k^n\,
\begin{pmatrix} f(k-1) \\ f(k-2) \\ \vdots \\ f(0) \end{pmatrix} \,.
\end{equation}
The value of $f(n+k-i)$ in left-hand column vector
is given by the dot product of the $i^{th}$ row of the matrix $C_k^n$ with the $k$-tuple of initial values $f(0)$ through $f(k-1)$.
For fixed $1 \le j \le k$,
if the initial $k$-tuple is $(f(0),\dots,f(k-1))=(\dots,\delta_{ij},\dots)$, where ``$\dots$'' denotes zeroes,
the solution for $f(n+k-i)$ is given by $c_{i,j}^{(n)}$, $1 \le i \le k$.
It is evident that the matrix elements in the $j^{th}$ column, $1 \le j \le k$,
are the elements of the $(k-j)^{th}$ basis sequence of the recurrence, viz.~$W_{k-j}(n-k-i)=c_{i,j}^{(n)}$, for $1 \le i \le k$.
\begin{proposition}
\label{thm:ident_cmpmtx}  
  Wolfram's Lemma 2.2 \cite{Wolfram2000}
  can be employed to express $c_{i,j}^{(n)}$ for $j=2,\dots,k$ in terms of the entries $c_{i,1}^{(\dots)}$
  in the first column of the companion matrix as follows
\begin{equation}
\label{eq:cijn_ident}  
c_{i,j}^{(n)} = \sum_{s=0}^{\min(k-j,n-1)} \beta_{j+s} c_{i,1}^{(n-s-1)} \qquad\qquad (j=2,\dots,k) \,.
\end{equation}
\end{proposition}
\begin{proof}
Transcribe the notation in eq.~\eqref{eq:W_basis}, accounting for differences in the indexing.
\end{proof}
Hence for $n>1$ it is only necessary to compute the entries in the first column of the companion matrix from scratch;
all the rest are linear combinations of those numbers.
This may be computationally helpful.
Eq.~\eqref{eq:cijn_ident} has been numerically tested and validated.
Wolfram's paper \cite{Wolfram2000} was published in 2000, i.e.~recently,
hence the above identity may not be known in the literature for companion matrices.
It is not mentioned by Chen and Louck \cite{ChenLouck}, whose paper is dated 1996.

%\newpage
\setcounter{equation}{0}
\section{Generating functions and multinomial sums}\label{sec:mult_sum}
\subsection{General remarks}
We remark on alternative expressions using generating functions and multinomial sums.
A referee kindly brought to the author's attention that
the multinomial sum in eq.~\eqref{eq:gen_multsum} below is
a Dickson polynomial of the second kind in several variables.
The connection to Dickson polynomials is shown below.
\begin{enumerate}
\item  
In this section, the arguments of all functions are nonnegative integers.
One can run all the recurrences backwards,
but the multinomial sums displayed below are valid only for nonnegative arguments.
\item  
To avoid ambiguity, we apply a circumflex, e.g.~$\hat{f}$, to denote the solution of a recurrence
when the initial $k$-tuple is $(0,\dots,0,1)$.
\end{enumerate}
  
\subsection{$k$-generalized Fibonacci}
Lee, Lee, Kim and Shin \cite{LeeLeeKimShin}
derived the Vandermonde set of equations, and published the formal solution, but did not find the roots explicitly.
Their last formula (unnumbered) expresses the $k$-generalized Fibonacci numbers using multinomial sums.
In terms of our notation, $F_k(0)=\dots=F_k(k-2)=0$, $F_k(k-1)=1$ and for $n\ge0$
\begin{equation}
\label{eq:Fnk_multsum}
\begin{split}
  \hat{F}_k(n+k-1) &= \sum_{i_1+2i_2+\dots+ki_k = n} \frac{(i_1+i_2+\dots+i_k)!}{i_1!i_2!\dots i_k!} 
  \\
  &= \sum_{i_1+2i_2+\dots+ki_k = n} \binom{i_1+i_2+\dots+i_k}{i_1,i_2,\dots, i_k} \,.
\end{split}
\end{equation}
The proof is given by employing the generating function for the $k$-generalized Fibonacci numbers 
\begin{equation}
G(z) = \frac{z^{k-1}}{1-z-z^2 -\dots -z^k} \,.
\end{equation}
Expand the denominator using the multinomial theorem and collect terms to obtain the coefficient of $z^{n+k-1}$.

\subsection{Linear recurrence with constant coefficients}\label{sec:more_gen}
We can generalize the recurrence and multinomial sum as follows.
We employ the recurrence in eq.~\eqref{eq:W_rec}.
The generating function is
\begin{equation}
G(z) = \frac{z^{k-1}}{1-\beta_1z-\dots-\beta_kz^k} \,.
\end{equation}
By construction, the initial values are $f(0)=\dots=f(k-2)=0$ and $f(k-1)=1$
and the remaining values of are given by the multinomial sum, for $n\ge0$,
\begin{equation}
\label{eq:gen_multsum}
\hat{f}(n+k-1) = \sum_{i_1+2i_2+\dots+ki_k = n} \binom{i_1+i_2+\dots+i_k}{i_1,i_2,\dots, i_k}\,\beta_1^{i_1}\dots\beta_k^{i_k} \,.
\end{equation}
\begin{remark}
Levesque \cite{Levesque1985} published a similar generating function and multinomial sums.
However, the numerator of Levesque's generating function incorporated arbitrary initial values for $f(0),\dots,f(k-1)$.
Our precomputed sums of linear combinations of the initial values in Sec.~\ref{sec:basis_seq} are similar to
Levesque's expressions for $v_i$ in (\cite{Levesque1985} eq.~(2.2)).
\end{remark}
A referee kindly brought to the author's attention that the sum in eq.~\eqref{eq:gen_multsum}
is a Dickson polynomial of the second kind in $k$ variables.
(The Dickson polynomials of the first and second kind, in one variable, are related to the Chebyshev polynomials.)
Observe that $\hat{f}(n+k-1)$ in eq.~\eqref{eq:gen_multsum} equals the matrix element $c_{1,1}^{(n)}$ (see eq.~\eqref{eq:cijn}).
The Dickson polynomial of the second kind in $k$ variables also equals $c_{1,1}^{(n)}$.
See eq.~(5.2) by Chen and Louck \cite{ChenLouck}, with the transcriptions $m\gets k-1$ and $u_i \gets \beta_i$.
From Prop.~\ref{thm:ident_cmpmtx}, all the other matrix elements $c_{i,j}^{(n)}$ are linear combinations of Dickson polynomials.

\subsection{Narayana sequence}
The Narayana sequence is given by the initial values $N(0)=N(1)=N(2)=1$ and for $n\ge3$ by the recurrence
\begin{equation}
\label{eq:Narayana_rec}  
N(n) = N(n-1) + N(n-3) \,.
\end{equation}
Then $k=3$ in eq.~\eqref{eq:Fnbeta} and the coefficients are
$(\beta_1,\beta_2,\beta_3)=(1,0,1)$.
We employ the basis functions $W_i(n)$ from Sec.~\ref{sec:basis_seq}
for this and the other sequences treated below.
The basis function with the initial values $(0,0,1)$ is $W^N_2(n)$. 
(We attach a superscript `N' for Narayana.)
The other basis functions are
$W^N_0(n)=\beta_3W^N_2(n-1)=W^N_2(n-1)$
and
$W^N_1(n)=\beta_2W^N_2(n-1)+\beta_3W^N_2(n-2) =W^N_2(n-2)$.
Using the given initial values, the solution is
\begin{equation}
\begin{split}
N(n) &= W^N_0(n) +W^N_1(n) +W^N_2(n)
\\
&= W^N_2(n-1) + \underbrace{W^N_2(n-2) +W^N_2(n)}_{=W^N_2(n+1)}
\\
&= W^N_2(n-1) +W^N_2(n+1)
\\
&= W^N_2(n+2) \,.
\end{split}
\end{equation}
For the multinomial sum, because $\beta_2=0$, we exclude all terms in $i_2$ in eq.~\eqref{eq:gen_multsum} with $i_2>0$.
This yields, using $i_1 = n-3i_3$,
\begin{equation}
W^N_2(n+2) = \sum_{i_1+3i_3= n} \binom{i_1+i_3}{i_1,i_3}
  = \sum_{i_3=0}^{\lfloor(n/3)\rfloor} \binom{n-2i_3}{i_3} \,.
\end{equation}
The first few values of $N(n)$ are $1,1,1,2,3,4,6,9,13,19,\dots$.

\subsection{Padovan sequence}\label{sec:pad}
The Padovan sequence is given by the initial values $P(0)=P(1)=P(2)=1$ and for $n\ge3$ by the recurrence
\begin{equation}
\label{eq:Padovan_rec}  
P(n) = P(n-2) + P(n-3) \,.
\end{equation}
Then $k=3$ in eq.~\eqref{eq:Fnbeta} and the coefficients are
$(\beta_1,\beta_2,\beta_3)=(0,1,1)$.
The basis function with the initial values $(0,0,1)$ is $W^P_2(n)$.
(We attach a superscript `P' for Padovan, and it can also mean Perrin in Sec.~\ref{sec:per} below.)
The other basis functions are
$W^P_0(n)=\beta_3W^P_2(n-1)=W^P_2(n-1)$
and
$W^P_1(n)=\beta_2W^P_2(n-1)+\beta_3W^P_2(n-2) =W^P_2(n-1)+W^P_2(n-2)$.
Using the given initial values, the solution is
\begin{equation}
\begin{split}
P(n) &= W^P_0(n) +W^P_1(n) +W^P_2(n)
\\
&= \underbrace{W^P_2(n-1) +W^P_2(n-2)}_{=W^P_2(n+1)} + \underbrace{W^P_2(n-1) +W^P_2(n)}_{=W^P_2(n+2)}
\\
&= W^P_2(n+1) +W^P_2(n+2)
\\
&= W^P_2(n+4) \,.
\end{split}
\end{equation}
For the multinomial sum,
because $\beta_1=0$, we exclude all terms in $i_1$ in eq.~\eqref{eq:gen_multsum} with $i_1>0$.
This yields
\begin{equation}
\label{eq:F_Padovan}
W^P_2(n+2) = \sum_{2i_2+3i_3=n} \binom{i_2+i_3}{i_2,i_3} \,.
\end{equation}
The first few values of $P(n)$ are $1,1,1,2,2,3,4,5,7,9,12,16,\dots$.

\subsection{Perrin sequence}\label{sec:per}
The Perrin and Padovan sequences satisfy the same recurrence, but with different initial values.
We write `$Q(n)$' to avoid confusion with the Padovan sequence $P(n)$.
The initial values are $Q(0)=3$, $Q(1)=0$ and $Q(2)=2$ and for $n\ge3$ the recurrence is
\begin{equation}
\label{eq:Perrin_rec}  
Q(n) = Q(n-2) + Q(n-3) \,.
\end{equation}
Because the recurrence eq.~\eqref{eq:Perrin_rec} is the same as eq.~\eqref{eq:Padovan_rec},
the basis functions are the same as in Sec.~\ref{sec:pad},
viz.~$W^P_i(n)$, $i=0,1,2$.
Using the given initial values, the solution is
\begin{equation}
\begin{split}
Q(n) &= 3W^P_0(n) +2W^P_2(n)
\\
&= 3W^P_2(n-1) +2W^P_2(n)
\\
&= 3\bigl(\,\underbrace{W^P_2(n-1) +W^P_2(n)}_{=W^P_2(n+2)}\,\bigr) -W^P_2(n)
\\
&= 3W^P_2(n+2) -W^P_2(n) \,.
\end{split}
\end{equation}
We can express the Perrin numbers as a linear combination of the Padovan numbers.
One choice (not unique) is $Q(n) = 4P(n) +2P(n+1) -3P(n+2)$.
The first few values of $Q(n)$ are $3,0,2,3,2,5,5,7,10,12,17,22,\dots$.

\section*{Acknowledgement}
The author is grateful to David Wolfram for numerous helpful comments on the manuscript
and especially for bringing Ref.~\cite{Wolfram2000} to his attention.
The author is also indebted to the referees for valuable suggestions for improvement,
including the connection to Dickson polynomials
and the reference to Chen and Louck \cite{ChenLouck}.

\section*{Disclosure statement}
The author declares there are no relevant financial or non-financial competing interests to report.

%\newpage
\bibliographystyle{amsplain}

\newpage
\begin{figure}[!htb]
\centering
\includegraphics[width=0.75\textwidth]{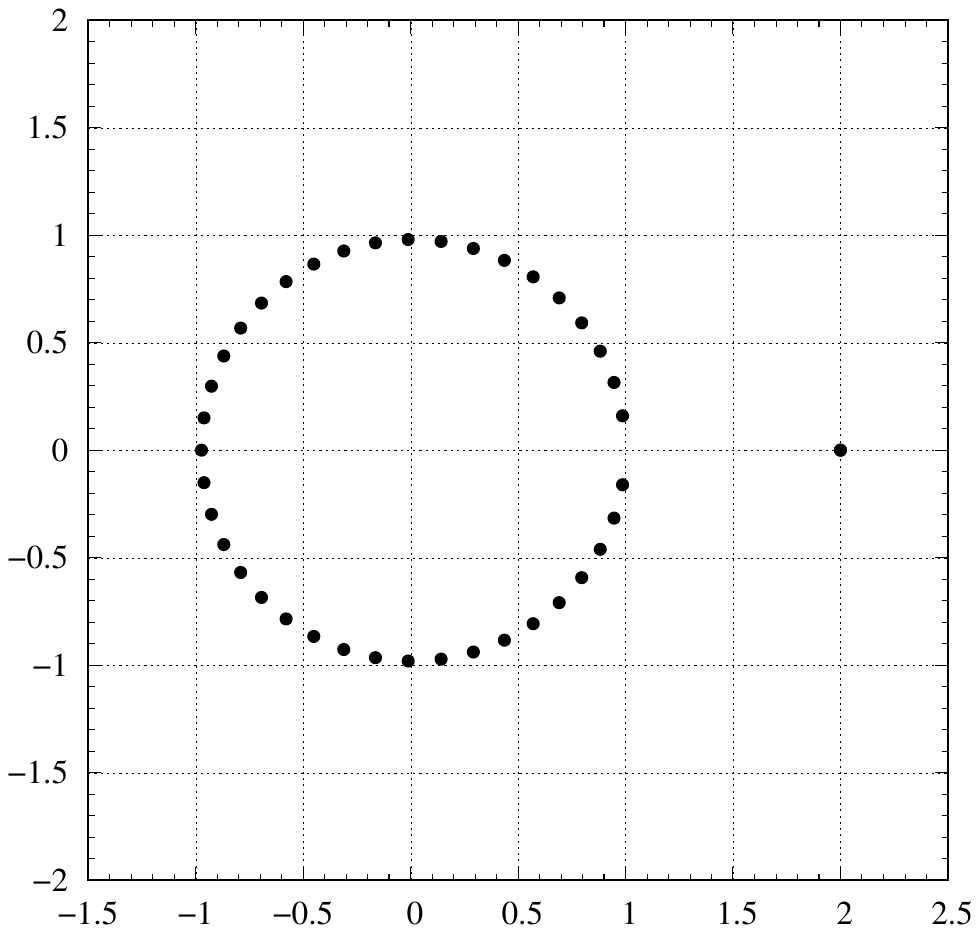}
\caption{\small
\label{fig:roots_k40}
Plot of the roots of the auxiliary polynomial in the complex plane, for $k=40$.}
\end{figure}

\newpage
\begin{figure}[!htb]
\centering
\includegraphics[width=0.75\textwidth]{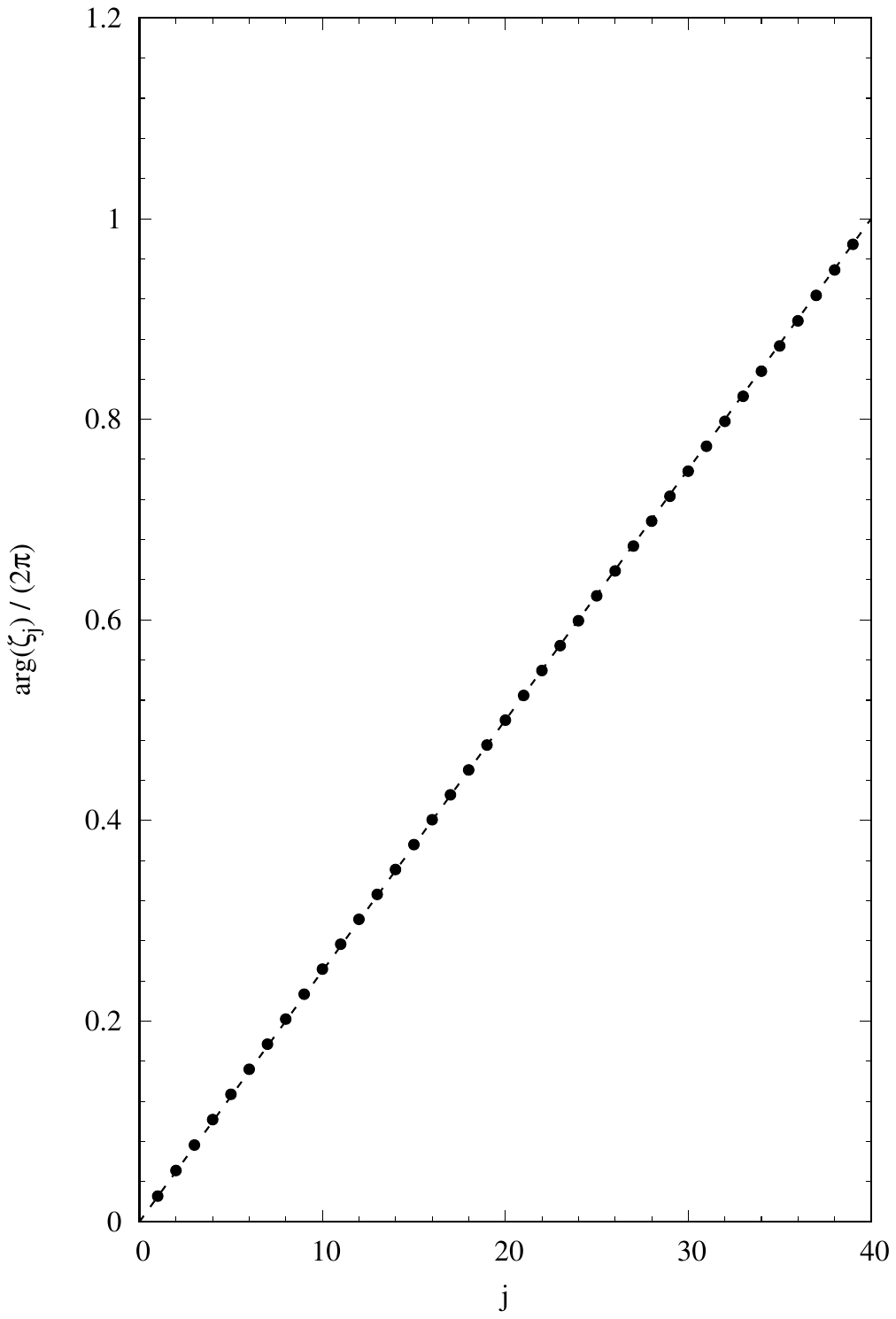}
\caption{\small
\label{fig:arg_roots_k40}
Plot of the argument of the secondary roots of the auxiliary polynomial 
$\arg(\zeta_j)/(2\pi)$, for $k=40$.}
\end{figure}

\end{document}